\documentclass[11pt]{CSML}
\usepackage[leqno]{amsmath}
\usepackage{amssymb, eucal}
\usepackage[all]{xy}
\xyoption{v2}
\DeclareMathAlphabet{\mathpzc}{OT1}{pzc}{m}{it}

\usepackage{lastpage}
\newcommand{\dOi}{13(3:32)2017}
\lmcsheading{}{1--\pageref{LastPage}}{}{}{Jan.~25,~2017}{Sep~27, 2017}{}

\usepackage{hyperref}
\hypersetup{hidelinks}

\theoremstyle{plain}
\newtheorem*{uprop}{Proposition}
\newtheorem*{ulem}{Lemma}
\newtheorem*{uthm}{Theorem}
\newtheorem*{ucor}{Corollary}
\theoremstyle{definition}
\newtheorem*{udefi}{Definition}

\newtheorem*{uexas}{Examples}

\theoremstyle{defC}
\newtheorem*{udefiC}{Definition}
\theoremstyle{thmC}
\newtheorem*{uthmC}{Theorem}
\newtheorem*{upropC}{Proposition}
\newtheorem*{ulemC}{Lemma}

\def\mathrmdef#1{\expandafter\def\csname#1\endcsname{{\rm#1}}}
\mathrmdef{Alg}\mathrmdef{b}\mathrmdef{Coalg}\mathrmdef{cod}\mathrmdef{Conv}\mathrmdef{dom}\mathrmdef{ev}\mathrmdef{coker}\mathrmdef{Down}\mathrmdef{false}\mathrmdef{Id}\mathrmdef{h}\mathrmdef{ker}\mathrmdef{l}\mathrmdef{max}
\mathrmdef{Ob}\mathrmdef{Inj}\mathrmdef{Emb}\mathrmdef{r}\mathrmdef{reg} \mathrmdef{op} \mathrmdef{true}\mathrmdef{Ult}

\def\mathsfdef#1{\expandafter\def\csname#1\endcsname{{\rm\sf#1}}}
\mathsfdef{A} \mathsfdef{App}
\mathsfdef{B}\mathsfdef{Bool} \mathsfdef{C}  \mathsfdef{Cat} \mathsfdef{CompHaus} \mathsfdef{D}
\mathsfdef{Gph} \mathsfdef{Met}\mathsfdef{Mod}\mathsfdef{Ord}\mathsfdef{PoSet}\mathsfdef{Rel} \mathsfdef{RelAlg}
\mathsfdef{S} \mathsfdef{Set} \mathsfdef{Top}

\def\V{V}
\def\TV{(\mT,\V)}
\def\VCat{V\mbox{-}\Cat}
\def\TVCat{(\mT,V)\mbox{-}\Cat}
\def\TVMod{(\mT,V)\mbox{-}\Mod}
\def\VRel{V\mbox{-}\Rel}
\def\two{\mathsf{2}}

\def\fx{\mathfrak{x}}
\def\fy{\mathfrak{y}}

\def\fw{\mathfrak{w}}

\def\cL{\mathcal{L}}

\def\cR{\mathcal{R}}

\def\fX{\mathfrak{X}}
\def\fY{\mathfrak{Y}}

\def\mL{\mathbb{L}}
\def\mP{\mathbb{P}}
\def\mR{\mathbb{R}}
\def\mS{\mathbb{S}}
\def\mT{\mathbb{T}}
\def\mU{\mathbb{U}}
\def\mId{\mathbb{I}\mathrm{d}}
\def\relto{{\longrightarrow\hspace*{-2.8ex}{\mapstochar}\hspace*{2.6ex}}}
\newcommand{\modto}{{\longrightarrow\hspace*{-2.8ex}{\circ}\hspace*{1.2ex}}}
\def\rel{\relto }

\newcommand{\multimapdot}{\mathrel{-\!\bullet}}
\newcommand{\multimapdotinv}{\mathrel{\bullet\!-}}
\newcommand{\homcompleft}{\multimapdot}
\newcommand{\homcompright}{\multimapdotinv}

\DeclareMathOperator{\yoneda}{\mathpzc{y}}
\DeclareMathOperator{\yonmult}{\mathpzc{m}}

\newcommand{\lari}{\textsc{lari}}

\begin{document}  
\title[Lax orthogonal factorisations in monad-quantale-enriched categories]{Lax orthogonal factorisations in
monad-quantale-enriched categories}  

\author{Maria Manuel Clementino}
\address{CMUC, Department of Mathematics\\ University of Coimbra\\ 3001-501 Coimbra\\ Portugal}
\email{mmc@mat.uc.pt}

\author{Ignacio L\'{o}pez-Franco}
\address{Department of Mathematics and Applications\\ CURE --
  Universidad de la Rep\'ublica\\
  Tacuaremb\'o s/n\\ Maldonado\\ Uruguay}
\email{ilopez@cure.edu.uy}
\dedicatory{Dedicated to Ji\v r\'{\i} Ad\'{a}mek}
\thanks{The authors
acknowledge partial financial  
assistance by Centro de Matem\'{a}tica da Universidade de
  Coimbra -- UID/MAT/00324/2013, funded by the Portuguese Government
  through FCT/MEC and co-funded by the European Regional Development Fund
  through the Partnership Agreement PT2020.}


\begin{abstract}
We show that, for a quantale $\V$ and a $\Set$-monad $\mathbb{T}$ laxly extended to $\V$-$\Rel$, the presheaf monad on
the category of $(\mT,\V)$-categories is simple, giving rise to a lax orthogonal factorisation system (\textsc{lofs})
whose corresponding weak factorisation system has embeddings as left part. In addition, we present presheaf submonads
and study the \textsc{lofs}s they define. This provides a method of constructing weak factorisation systems on some well-known examples of topological categories over $\Set$.

\end{abstract}
\maketitle

\section{Introduction}
In 1985 Cassidy-H\'{e}bert-Kelly \cite{CHK} studied orthogonal factorisations systems induced by reflective subcategories,
with particular emphasis in the case when the reflection is simple. Among the lax orthogonal factorisation systems (\textsc{lofs}s), that
generalise the orthogonal ones in 2-categories, those arising from simple monads -- as defined by the authors of this
paper in \cite{CLF16, CLF17} -- have particular relevance. This paper intends to give a systematic way of producing
simple monads in (some) topological categories over \Set\ using the presheaf monads of $(\mT,\V)$-Cat studied in
\cite{H11, CH08}. Given a quantale $\V$ and a well-behaved \Set-monad $\mT$, the category $\TV$-\Cat, of generalised
$\V$-enriched categories and their functors, is topological and locally preordered (see \cite{CH03, CT03}). As crucial
examples we mention the categories \Ord\ of (pre)ordered sets and monotone maps, \Top\ of topological spaces and
continuous maps, \Met\ of Lawvere generalised metric spaces and non-expansive maps \cite{Law73}, and \App\ of Lowen
approach spaces and non-expansive maps \cite{Lo}.
Equipping the quantale $\V$ with a canonical $\TV$-category structure, one gets naturally a Yoneda Lemma and a
well-behaved presheaf monad that was shown to be lax idempotent in \cite{H11}. Here we show that it is simple, inducing a
lax orthogonal factorisation system which underlies a weak factorisation system having embeddings as left part. (In order
to avoid technicalities we restrict ourselves to separated, or skeletal, $\TV$-categories, so that their hom-sets have an
anti-symmetric order.) This encompasses the weak factorisation system in \Ord\ studied by
Ad\'{a}mek-Herrlich-Rosick\'y-Tholen in \cite{AHRT}.

These presheaf monads have interesting simple submonads, namely the ones which have as algebras the Lawvere complete
$\TV$-categories (see \cite{CH09}), and that gives -- as shown by Lawvere in \cite{Law73} -- Cauchy-complete spaces when
one takes $\mT=\mId$ the identity monad and $\V$ the complete half-real line. These also cover, following techniques
developed in \cite{CH08}, the weak factorisation systems of $\Top_0$ studied in \cite{CCM}, having as left parts
embeddings, dense embeddings, flat embeddings and completely flat embeddings.

The examples of \textsc{lofs}s and weak factorisation systems we consider are the result of a general construction distinct from, and in many ways orthogonal to, the more usual method of cofibrant generation. The construction of cofibrantly generated weak factorisation systems is usually known as Quillen's small object argument \cite{Qui}. A version for algebraic factorisation systems was introduced in \cite{Ga}. The construction we employ, introduced in \cite{CLF16}, gives rise to weak factorisations systems that need not be cofibrantly generated (see \cite{LF}).

This paper does not intend to be self-contained. In Section 2 and 3 we present the basic definitions and results on lax
orthogonal factorisation systems and on $\TV$-categories that are needed for this work. For a better understanding of
these topics we refer to the papers mentioned there and to the monograph \cite{HST}. In Section 4 we study the presheaf
monads on $\TV$-categories and their simplicity. In Section 5 we explore the examples of lax orthogonal factorisation
systems induced by these presheaf monads.

\section{Lax orthogonal factorisation systems}
Throughout we will be working in a category \C\ enriched in posets, or \emph{$\Ord$-enriched category}, so that each
hom-set $\C(X,Y)$ is equipped with an order structure $\leq$ that is preserved by composition: if $f,f':X\to Y$, with
$f\leq f'$, $g:Y\to Z$ and $h:W\to X$, then $g\cdot f\leq g\cdot f'$ and $f\cdot h\leq f'\cdot h$.

\subsection{Weak factorisation systems}
Given morphisms $f,g$, we say that \emph{$f$ has the left lifting property with respect to
$g$}, and that \emph{$g$ has the right lifting property with respect to $f$}, if
every commutative square as shown has a (not
necessarily unique) diagonal filler.
\[\xymatrix{.\ar[r]\ar[d]_f&.\ar[d]^g\\
.\ar[r]\ar@{.>}[ru]&.}\]
A \emph{weak factorisation system} (\textsc{wfs}) in a category is a pair $(\cL,\cR)$ of
families of morphisms such that:
\begin{itemize}
\item $\mathcal{L}$ consists of those morphisms with the left lifting property
  with respect to each morphism of $\mathcal{R}$.
\item $\mathcal{R}$ consists of those morphisms with the right lifting property
  with respect to each morphism of $\mathcal{L}$.
\item Each morphism in the category factors through an element
  of $\mathcal{L}$ followed by one of $\mathcal{R}$.
\end{itemize}

\subsection{Algebraic weak factorisation systems}
 An \Ord-functorial factorisation on an \Ord-category \C\ consists of a factorisation
 $\dom\stackrel{\lambda}{\Rightarrow}E\stackrel{\rho}{\Rightarrow}\cod$
  of the natural
  transformation $\dom\Rightarrow\cod$ with component at $f\in \C^\two$ equal to
  $f\colon\dom(f)\to\cod(f)$, in the category of locally monotone functors $\C^\two\to \C$.
  As in the case of functorial factorisations on ordinary categories, an
  \Ord-functorial factorisation
  can be equivalently described as:
  \begin{itemize}
  \item A copointed endo-\Ord-functor $\Phi\colon L\Rightarrow 1_{\C^\two}$ on
    $\C^\two$ with $\dom(\Phi)=1$.
  \item A pointed endo-\Ord-functor $\Lambda\colon 1_{\C^\two}\Rightarrow R$ on
    $\C^\two$ with $\cod(\Lambda)=1$.
  \end{itemize}
  The three descriptions of an \Ord-functorial factorisation are related by:
  \[    \dom(\Lambda_f)=Lf=\lambda_f\qquad \cod(\Phi_f)=Rf=\rho_f.
  \]

  An \emph{algebraic weak factorisation system}, abbreviated {\textsc{awfs}}, on an
  \Ord-category \C\ consists of a pair $(\mL,\mR)$, where
  $\mL=(L,\Phi,\Sigma)$ is an \Ord-comonad and
  $\mR=(R,\Lambda,\Pi)$ is an \Ord-monad on $\C^\two$, such that $(L,\Phi)$ and
  $(R,\Lambda)$ represent the same \Ord-functorial factorisation on \C\ (i.e., the
  equalities above hold), fulfilling a distributivity condition which we explain next.

 Note that the components of $\Sigma$ and $\Pi$ are as follows
 \[
   \Sigma_f=
   \diagram
   \cdot{}\ar@{=}[r]\ar[d]_{Lf}&
   \cdot{}\ar[d]^{L^2f}\\
   \cdot{} \ar[r]^{\sigma_f}
   & \cdot{}
   \enddiagram
   \quad\text{ and }\quad
   \Pi_f=
   \diagram
   \cdot{}\ar[r]^{\pi_f}\ar[d]_{R^2f}&\cdot{}\ar[d]^{Rf}\\
     \cdot{}\ar@{=}[r]&\cdot{}
     \enddiagram
 \]
  which together form a natural transformation $\Delta:LR\Rightarrow RL$, with $\Delta_f=(\sigma_f,\pi_f)$ as below.
  \[\Delta_f=
    \diagram
    Ef\ar[r]^{\sigma_f}\ar[d]_{LRf}\ar@{.>}[rd]^1&ELf\ar[d]^{RLf}\\
    ERf\ar[r]_{\pi_f}&Ef
    \enddiagram
  \]
  The distributivity axiom requires $\Delta$ to be a mixed distributive law between
  the comonad $\mL$ and the monad $\mR$, that reduces to the
  commutativity of the following diagrams.
  \begin{equation}
    \label{distrib}
    \begin{split}
      \xymatrix{LR^2\ar[r]^{\Delta R}\ar[d]_{L\Pi}&RLR\ar[r]^{R\Delta}&R^2L\ar[d]^{\Pi L}&LR\ar[rr]^{\Delta}\ar[d]_{\Sigma
          R}&&RL\ar[d]^{R\Sigma}\\
        LR\ar[rr]^{\Delta}&&RL&L^2R\ar[r]^{L\Delta}&LRL\ar[r]^{\Delta L}&RL^2}
    \end{split}
  \end{equation}
Algebraic weak factorisation systems were introduced by Grandis-Tholen in \cite{GT} under the name \emph{natural
factorisation system}; later, in \cite{Ga}, Garner added to this definition the distributivity conditions we described
above.

  Each \textsc{awfs} has an underlying \textsc{wfs} $(\mathcal{L},\mathcal{R})$, with $\mathcal{L}=\{f\,|\, f$ has an
  $(L,\Phi)$-coalgebra structure$\}$ and $\mathcal{R}=\{g\,|\, g$ has an $(R,\Lambda)$-algebra structure$\}$.
  A coalgebra structure $(1_X,s:Y\to Ef)$ for $f\in\cL$, so that $s\cdot f=Lf$ and $Rf\cdot s=1_Y$, and an
  $(R,\Lambda)$-algebra structure $(p:Eg\to Z,1_W)$ for $g\in\cR$, so that $g\cdot p=Rg$ and $p\cdot Lg=1_Z$, give a
  natural lifting $d=p\cdot E(u,v)\cdot s$ for a commutative square $v\cdot f=g\cdot u$:
  \begin{equation}
    \label{lift}
    \begin{split}
      \xymatrix{X\ar[rr]^u\ar[dd]_f\ar[rd]^{Lf}&&Z\ar@{=}[r]&Z\ar[dd]^g\\
        &Ef\ar[r]^{E(u,v)}&Eg\ar[ru]^p\ar[rd]^{Rg}\\
        Y\ar@{=}[r]\ar[ru]^s&Y\ar[rr]^v&&W.}
    \end{split}
  \end{equation}
  This lifting is unique -- so that $(\cL,\cR)$ is an \emph{orthogonal factorisation system} -- if, and only if, $\mL$
  and $\mR$ are idempotent. In fact idempotency of $\mL$ implies idempotency of $\mR$ and vice-versa, as shown in
  \cite{BG}.

\subsection{Lax orthogonal factorisation systems}
Informally, a lax orthogonal factorisation system is an \textsc{awfs} whose liftings as in \eqref{lift} have a universal
property, as we explain next. First we recall that:
\begin{udefiC}[\cite{Ko}]
An \Ord-enriched monad $\mS=(S,\eta,\mu)$ is \emph{lax idempotent}, or \emph{Kock-Z\"oberlein}, if it satisfies any of
the following equivalent conditions:
\begin{enumerate}[label=(\roman*)]
\item $S\eta\leq\eta S$;
\item $S\eta\dashv \mu$ (or, equivalently, $S\eta\cdot\mu\leq 1$);
\item $\mu\dashv \eta S$ (or, equivalently, $1\leq \eta S\cdot\mu$);
\item a morphism $f:SX\to X$ is an $\mS$-algebra structure for $X$ if, and only if, $f\dashv \eta_X$ with
    $f\cdot\eta_X=1_X$.
\end{enumerate}
A lax idempotent $\Ord$-comonad is defined dually.
\end{udefiC}

\begin{ulem}
If $\mS=(S,\eta,\mu)$ is a lax idempotent monad on an \Ord-category, the following conditions are equivalent, for an
object $X$ of $\C$:
\begin{enumerate}[label=(\em\roman*)]
\item $X$ admits an $\mS$-algebra structure;
\item $X$ admits a unique $\mS$-algebra structure;
\item $\eta_X:X\to SX$ has a right inverse, i.e. $X$ admits an $(S,\eta)$-algebra structure;
\item $X$ is a retract of $SX$;
\item $X$ is a retract of an $\mS$-algebra.
\end{enumerate}
\end{ulem}

An \textsc{awfs} $(\mL,\mR)$ is a \emph{lax orthogonal factorisation system}, abbreviated \textsc{lofs}, if $\mL$ and
$\mR$ are lax idempotent. These factorisations were introduced by the authors in \cite{CLF16} and further studied in the
\Ord-enriched categories setting, as used here, in \cite{CLF17}.
\begin{ucor}
If $(\mL,\mR)$ is a \textsc{lofs}, then its underlying weak factorisation system $(\cL,\cR)$ consists of the class $\cL$
of the morphisms admitting a (unique) $\mL$-coalgebra structure and the class $\cR$ consists of the morphisms admitting a
(unique) $\mR$-algebra structure.
\end{ucor}

As for orthogonal factorisation systems, lax idempotency of $\mL$ implies lax idempotency of $\mR$ and vice-versa. In
fact:
\begin{uthmC}[\cite{CLF17}]\leavevmode
\begin{enumerate}[label=(\arabic*)]
\item Given an \textsc{awfs} $(\mL,\mR)$ on an \Ord-category \C, the following conditions are equivalent:
\begin{enumerate}[label=(\em\roman*)]
\item $(\mL,\mR)$ is a \textsc{lofs};
\item $\mL$ is lax idempotent;
\item $\mR$ is lax idempotent.
\end{enumerate}
\item Given a domain-preserving \Ord-comonad $\mL$ and a codomain-preserving \Ord-monad $\mR$ inducing the same
    \Ord-functorial factorisation $f=Rf\cdot Lf$, the following conditions are equivalent:
\begin{enumerate}[label=(\em\roman*)]
\item $(\mL,\mR)$ is a \textsc{lofs}.
\item Both $\mL$ and $\mR$ are lax idempotent.
\item One of $\mL$ and $\mR$ is lax idempotent and the distributive law axioms \eqref{distrib} hold.
\end{enumerate}
\end{enumerate}
\end{uthmC}

\subsection{Lifting operations}
Diagram \eqref{lift} shows that every functorial factorisation system induces a canonical lifting operation from the
forgetful \Ord-functor $U:(L,\Phi)$-$\Coalg\to\C^2$ to the forgetful \Ord-functor $V:(R,\Lambda)$-$\Alg\to\C^2$, meaning
that every commutative diagram
\begin{equation}
\label{liftop}
\begin{split}
\xymatrix{.\ar[r]^h\ar[d]_{Ua}&.\ar[d]^{Vb}\\
.\ar[r]^k&.}
\end{split}
\end{equation}
has a canonical diagonal filler $\phi_{a,b}(h,k)$ so that $Vb\cdot\phi_{a,b}(h,k)=k$, $\phi_{a,b}(h,k)\cdot Ua=h$. Those
fillers respect both composition and order in a natural way (see \cite{CLF17} for details).

A lifting operation from $U:\A\to\C^2$ to $V:\B\to\C^2$ is said to be \textsc{kz} if, for every commutative diagram
\eqref{liftop} and every diagonal filler $d$, one has $\phi_{a,b}(h,k)\leq d$.
\begin{uthmC}[\cite{CLF17}]
For an \textsc{awfs} $(\mL,\mR)$ on an \Ord-category \C, the following conditions are equivalent:
\begin{enumerate}[label=(\em\roman*)]
\item $(\mL,\mR)$ is a \textsc{lofs}.
\item The lifting operation from $\mL$-$\Coalg\to\C^2$ to $\mR$-$\Alg\to\C^2$ is \textsc{kz}.
\end{enumerate}
\end{uthmC}

\subsection{Simple monads and their {\normalfont\textsc{lofs}s}}\label{sub:2.5}
The notion of simple monad we present here, studied in \cite{CLF16, CLF17}, is the \Ord-enriched version of simple
reflection of \cite{CHK}. In an $\Ord$-enriched category $\C$ with comma-objects, given an \Ord-monad $\mS=(S,\eta,\mu)$, we construct a monad $\mR$ on $\C^2$ by considering the
comma-object $Kf=Sf\downarrow\eta_Y$ and defining $Rf:Kf\to Y$ as the second projection. Then $Lf$ is the unique morphism
making the following diagram commute.
\[\xymatrix{X\ar[rdd]_f\ar[rd]|{Lf}\ar[rrd]^{\eta_X}\\
&Kf\ar[r]|{q_f}\ar[d]|{Rf}&SX\ar[d]^{Sf}\\
&Y\ar@{}[ru]|{\geq}\ar[r]_{\eta_Y}&SY}\]
The \Ord-functorial factorisation $f=Rf\cdot Lf$ defines a copointed endo-$\Ord$-functor $(L,\Phi:L\Rightarrow 1)$, with
$\Phi_f=(1_X,Rf)$, and a pointed endo-$\Ord$-functor $(R,\Lambda)$, with $\Lambda_f=(Lf,1_Y)$. Moreover, $(R,\Lambda)$
underlies a monad $\mR$ on $\C^2$ whose multiplication $\Pi_f=(\pi_f,1_Y)$ is defined by the unique morphism $\pi_f$
given by the universal property of the comma-object:
\[\xymatrix{KRf\ar[rd]^{\pi_f}\ar[rr]^{q_{Rf}}\ar[rdd]_{RRf}&&SKf\ar[rr]^{Sq_f}&&SSX\ar[d]^{\mu_X}\\
&Kf\ar[rrr]^{q_f}\ar[d]|{Rf}&&&SX\ar[d]^{Sf}\\
&Y\ar[rrr]_{\eta_Y}\ar@{}[rrru]|{\geq}&&&SY}\]
(See \cite{CLF17} for details.)

\begin{ulem}
Given a monad $\mS$ on $\C$, the following conditions are equivalent for a morphism $f:X\to Y$ in \C:
\begin{enumerate}[label=(\em\roman*)]
\item $f$ has an $(L,\Phi)$-coalgebra structure.
\item $Sf$ is a {\lari} (=\emph{l}eft \emph{a}djoint \emph{r}ight \emph{i}nverse\footnote{We use the terminology introduced by J.W. Gray in \cite{Gr}.}), that is, it has a right adjoint
    $S^*f$ such that $S^*f\cdot Sf=1$.
\end{enumerate}
\end{ulem}
\begin{proof}
(i)$\Rightarrow$(ii): If $(1_X,s:Y\to Kf)$ is an $(L,\Phi)$-coalgebra structure for $f:X\to Y$, then $S^*f:=\mu_X\cdot
Sq_f\cdot Ss$ is a left inverse of $Sf$:
\[S^*f\cdot Sf=\mu_X\cdot Sq_f\cdot Ss\cdot Sf=\mu_X\cdot Sq_f\cdot SLf=\mu_X\cdot S\eta_X=1;\]
and, moreover, it is right adjoint to $Sf$:
\[Sf\cdot S^*f=Sf\cdot\mu_X\cdot Sq_f\cdot Ss=\mu_Y\cdot SSf\cdot Sq_f\cdot Ss\leq \mu_Y\cdot S\eta_Y\cdot SRf\cdot Ss=
1.\]

(ii)$\Rightarrow$(i): Let $S^*f$ be a right adjoint left inverse of $Sf$. By definition of comma-object, from $Sf\cdot
(S^*f\cdot\eta_Y)\leq \eta_Y$ there exists a unique $s:Y\to Kf$ such that $Rf\cdot s=1_Y$ and $q_f\cdot
s=S^*f\cdot\eta_Y$. To conclude that $s\cdot f=Lf$, compose $s\cdot f$ with $Rf$ and $q_f$:
\[Rf\cdot s\cdot f=f\mbox{ and }q_f\cdot s\cdot f=S^*f\cdot\eta_Y\cdot f=S^*f\cdot Sf\cdot\eta_X=\eta_X.\tag*{\qEd}\]
\def\popQED{}
\end{proof}

A morphism $f$ in \C\ such that $Sf$ is a \lari{} is called an \emph{$S$-embedding}. Denote by $\mS$-$\Emb$ the category
that has as objects pairs $(f,r)$ of morphisms of \C\ such that $Sf\dashv r$ with $r\cdot Sf=1$, and as morphisms
$(h,k):(f,r)\to (g,s)$ morphisms $(h,k):f\to g$ in $\C^2$ such that $s\cdot Sk=Sh\cdot r$.

\begin{udefi}
The \Ord-monad $\mS$ is said to be \emph{simple} if the locally monotone forgetful functor $\mS$-$\Emb\to\C^2$ has a
right adjoint and the induced comonad has underlying functor $L$ and counit $\Phi$.
\end{udefi}

As shown in \cite{CLF17}:
\begin{uprop}
A lax idempotent monad $\mS=(S,\eta,\mu)$ on \C\ is simple if, and only if, for every morphism $f:X\to Y$, there is an
adjunction
\[SLf\dashv \mu_X\cdot Sq_f.\]
\end{uprop}

\begin{uthm}
If $\mS$ is a lax idempotent and simple monad, then $(\mL,\mR)$ is a \textsc{lofs}. Moreover, the left class $\cL$ of the
weak factorisation system it induces is the class of $S$-embeddings.
\end{uthm}
\begin{proof}[Proof (Sketch of the proof; for details see \cite{CLF17}).]
Simplicity of $\mS$ gives the comonad structure for $L$ needed to define the \textsc{awfs}.

In order to show that $\mR$ is a lax idempotent monad, that is, $R\Lambda_f\leq\Lambda_{Rf}=(LRf,1_Y)$, we denote
$R\Lambda_f$ by $((R\Lambda_f)_1,1_Y)$ and note that, by definition of $R$, $R^2f\cdot (R\Lambda_f)_1=Rf$ and
$q_{Rf}\cdot (R\Lambda_f)_1=SLf\cdot q_f$. Then
\[q_{Rf}\cdot (R\Lambda_f)_1=\mu_{SKf}\cdot\eta_{S^2Kf}\cdot SLf\cdot q_f=SLf\cdot\mu_X\cdot Sq_f\cdot \eta_{Kf}\leq
\eta_{Kf},\]
by simplicity of $S$. Now, by definition of comma-object and by the equalities $R^2f\cdot LRf=Rf$ and $q_f\cdot
LRf=\eta_{Kf}$, it follows that $(R\Lambda_f)_1\leq LRf$ as claimed.

The last assertion follows from the lemma above.
\end{proof}

\subsection{Submonads of simple monads}\label{sub:2.6}
Well-behaved submonads of simple monads are simple, as stated below.
\begin{uthmC}[\cite{CLF17}]
Suppose that $\varphi:\mS'\to \mS$ is a monad morphism between \Ord-monads whose components are pullback-stable
$\mS$-embeddings, and that $\mS$-embeddings are full. If $\mS$ is lax idempotent, then $\mS'$ is simple whenever $\mS$ is
so. Moreover, every $\mS'$-embedding is an $\mS$-embedding.
\end{uthmC}
(Here by \emph{full morphism} in an \Ord-category we mean a morphism $f:X\to Y$ such that, for every $u,v:Z\to X$,
$f\cdot u\leq f\cdot v$ implies $u\leq v$.)

\section{$\TV$-categories}\label{sect3}
\subsection{The setting}\label{sub:3.1}
First we describe the setting where we will be working throughout the paper.

\noindent A. $\V$ is a \emph{commutative and unital quantale}, that is, a complete lattice equipped with a tensor product
$\otimes$, with unit $k\neq\bot$, and with right adjoint $\hom$. We denote by $\VRel$ the bicategory of
\emph{$\V$-relations}, having sets as objects, while morphisms $r:X\relto Y$ are $\V$-relations, i.e. maps $r:X\times
Y\to\V$; their composition is given by relational composition, that is, for $r:X\relto Y$ and $s:Y\relto Z$,
\[s\cdot r(x,z)=\bigvee_{y\in Y}r(x,y)\otimes s(y,z).\]
Every map $f:X\to Y$ \emph{is} a $\V$-relation $f:X\times Y\to\V$ with $f(x,y)=k$ if $f(x)=y$ and $f(x,y)=\bot$
elsewhere. This correspondence defines a bijective on objects and faithful pseudofunctor $\Set\to\VRel$. $\VRel$ is a
locally ordered and locally complete bicategory, with $r\leq s$ if $r(x,y)\leq s(x,y)$, for $r,s:X\relto Y$, $x\in X$,
$y\in Y$. It has an involution $(\;)^\circ:\VRel\to\VRel$ assigning to each $r:X\relto Y$ the $\V$-relation
$r^\circ:Y\relto X$ defined by $r^\circ(y,x)=r(x,y)$. For each $r:X\relto Y$ both left and right compositions with $r$
preserve suprema, and therefore we have the following adjunctions
\[\xymatrix{\VRel(Y,Z)\ar@{}[rr]|-{\bot}\ar@<1ex>[rr]^-{(\;)\cdot r}&&\VRel(X,Z)\ar@<1ex>[ll]^-{(\;)\homcompright
r}}\mbox{ and }\xymatrix{\VRel(Z,X)\ar@{}[rr]|-{\bot}\ar@<1ex>[rr]^-{r\cdot(\;)}&&\VRel(Z,Y)\ar@<1ex>[ll]^-{r\homcompleft
(\;)}}\]
so that, for every $s:Y\relto Z$, $s':X\relto Z$, $t:Z\relto X$, $t':Z\relto Y$,
\[s\cdot r\leq s'\iff s\leq s'\homcompright r\mbox{ and }r\cdot t\leq t'\iff t\leq r\homcompleft t'.\]

\noindent{B.} $\mT=(T,e,m)$ is a non-trivial $\Set$-monad that \emph{satisfies (BC)}; that is, $T$ preserves weak
pullbacks and every naturality square of $m$ is a weak pullback. We point out that, in particular, the monad $\mT$ is
\emph{taut} in the sense of Manes \cite{Ma02} (see \cite{CHJ14} for details).

\noindent{C.} $\xi:T\V\to\V$ is a $\mT$-algebra structure on $\V$ such that both $\otimes:\V\times\V\to\V$ and $k:1\to
V$, $\ast\mapsto k$, are $\mT$-algebra homomorphisms, that is, the following diagrams
\[\xymatrix{T1\ar[r]^{Tk}\ar[d]_{!}&T\V\ar[d]^\xi&&&T(\V\times\V)\ar[r]^-{T(\otimes)}\ar[d]_{<\xi\cdot T\pi_1,\xi\cdot
T\pi_2>}&T\V\ar[d]^\xi\\
1\ar[r]^k&\V&&&\V\times\V\ar[r]^-{\otimes}&\V}\]
are commutative, and, for all maps $f:X\to Y$, $\varphi:X\to\V$ and $\psi:Y\to\V$ with
$\displaystyle\psi(y)=\bigvee_{x\in f^{-1}(y)}\varphi(x)$ for every $y\in Y$, the following inequality holds
\[\xi\cdot T\psi(y)\leq\bigvee_{\fx\in Tf^{-1}(\fy)}\xi\cdot T\varphi(\fx),\]
for every $\fy\in TY$. (For alternative descriptions of the latter condition see \cite{H07}.)

\noindent{D.} Using $\xi$ we define, for each $\V$-relation $r:X\rel Y$, the $\V$-relation $T_\xi r:TX\rel TY$ as the
composite
\[\xymatrix{TX\times TY\ar[r]^-{T_\xi r}\ar[d]|-{\object@{|}}_{<T\pi_1,T\pi_2>^\circ}&\V\\
T(X\times Y)\ar[r]_-{Tr}&T\V\ar[u]_\xi}\]
that is, for each $\fx\in TX$, $\fy\in TY$,
\[T_\xi r(\fx,\fy)=\bigvee\{\xi(Tr(\fw))\,|\,\fw\in T(X\times Y), T\pi_1(\fw)=\fx, T\pi_2(\fw)=\fy\}.\]
This defines a pseudofunctor $T_\xi:\VRel\to\VRel$ that extends $T:\Set\to\Set$, so that $m:T_\xi T_\xi\to T_\xi$ is a
natural transformation while $e:\Id_{\VRel}\to T_\xi$ is an op-lax natural transformation (see \cite{H07} for details).

\subsection{$\TV$-categories}\label{sub:3.2}
Having fixed these data, a $\TV$-category is a pair $(X,a)$, where $X$ is a set and $a:TX\rel X$ is a $\V$-relation such
that
\[1_X\leq a\cdot e_X\mbox{ and }a\cdot Ta\leq a\cdot m_X.\]
Given $\TV$-categories $(X,a)$, $(Y,b)$, a $\TV$-functor $f:(X,a)\to(Y,b)$ is a map $f:X\to Y$ such that
\[f\cdot a\leq b\cdot Tf.\]
We denote the category of $\TV$-categories and $\TV$-functors by $\TVCat$. As defined in \cite[Section 12]{CT03},
$\TVCat$ is (pre)order-enriched by:
\[f\leq g\mbox{ if }g\leq b\cdot e_Y\cdot f,\]
for $f,g:(X,a)\to(Y,b)$. (This structure is in fact inherited from the order-enrichment of $\VRel$ as explained in
\ref{sub:3.5}.) Identifying an element $x$ of $X$ with the $\TV$-functor $E=(1,e_1^\circ)\to (X,a)$, $\ast\mapsto x$,
$(X,a)$ becomes (pre)ordered; $(X,a)$ is called \emph{separated}, or \emph{skeletal}, if, for $x,x'\in X$, $x\leq x'$ and
$x'\leq x$ implies $x=x'$. The category of separated $\TV$-categories and $\TV$-functors will be denoted by $\TVCat_0$.

\begin{uexas}
Let $\mT$ be the identity monad $\mId$ and $\xi:V\to V$ the identity map.
\begin{itemize}
\item When $\V=\two$, $(\mId,\two)$-\Cat\ is the category of (pre)ordered sets and monotone maps.
\item Let $\V=[0,\infty]_+$ be the complete half-real line $[0,\infty]$ ordered by the greater or equal relation, with
    $\otimes=+$ and $\hom$ the truncated minus, so that $\hom(u,v)=v\ominus u$, which is equal to $v-u$ if $v\geq u$
    and $0$ otherwise. As shown by Lawvere in \cite{Law73}, $(\mId,[0,\infty]_+)$-\Cat\ is the category $\Met$ of generalised
    metric spaces and non-expansive maps.
\end{itemize}
Let $\mT$ be the ultrafilter monad $\mU$ and $\xi:TV\to V$ be defined by $\xi(\fx)=\bigvee\{v\in \V\,|\,\fx\in T(\uparrow
v)\}$.
\begin{itemize}
\item When $\V=\two$ -- as shown by Barr in \cite{Ba} -- $(\mU,\two)$-\Cat\ is the category $\Top$ of topological spaces and
    continuous maps.
\item When $\V=[0,\infty]_+$ -- as shown in \cite{CH03} -- $(\mU,[0,\infty]_+)$-\Cat\ is the category of approach
    spaces and non-expansive maps \cite{Lo}.
\end{itemize}
\end{uexas}
\subsection{The dual of a $\TV$-category}
When $\mT$ is the identity monad, $\TVCat$ is the category $\VCat$ of $\V$-categories and $\V$-functors. In $\VCat$ there
is a natural notion of \emph{dual category}, inducing a functor $D:\VCat\to\VCat$, with $D(X,a)=(X,a^\circ)$. To build a
dual for a $\TV$-category we first note that the $\Set$-monad $\mT$ can be extended to $\VCat$, with $T(X,a)=(TX,Ta)$,
and make use of the following adjunction
\[\xymatrix{(\VCat)^\mT\ar@{}[rr]|{\bot}\ar@<1ex>[rr]^-N&&\TVCat\ar@<1ex>[ll]^-M}\]
where, for a $\V$-category $(X,a)$, a $\mT$-algebra structure $\alpha:T(X,a)\to(X,a)$ and a $\mT$-homomorphism $f$,
$N((X,a),\alpha)=(X,a\cdot\alpha)$ and $Nf=f$; and, for a $\TV$-category $(Y,b)$ and a $\TV$-functor $g$,
$M(Y,b)=((TY,Tb\cdot m_Y^\circ),m_Y)$ and $Mg=Tg$. The functor $D:\VCat\to\VCat$ lifts to a functor
$D:(\VCat)^\mT\to(\VCat)^\mT$, with $D((X,a),\alpha)=((X,a^\circ),\alpha)$, and we define the \emph{dual $(X,a)^\op$ of a
$\TV$-category} $(X,a)$ as
\[\xymatrix{(\VCat)^\mT\ar@`{(-20,-20),(-20,20)}^D\ar@{}[rr]|{\bot}\ar@<1ex>[rr]^-N&&\TVCat\ar@<1ex>[ll]^-M}\]
$NDM(X,a)=(TX, m_X\cdot (Ta)^\circ\cdot m_X)$; that is, denoting its structure by $a^\op$,
\[a^\op(\fX,\fy)=\bigvee_{\fY\,:\,m_X(\fY)=\fy}Ta(\fY,m_X(\fX)),\]
for $\fX\in T^2X$ and $\fy\in TX$ (see \cite{CCH}).

\subsection{$\V$ as a $\TV$-category}
As we have in $\V$ both a $\V$-categorical structure $\hom:\V\rel\V$ and a $\mT$-algebra structure $\xi:T\V\to\V$, which
is a $\V$-functor $\xi:(T\V,T\hom)\to (\V,\hom)$ due to our assumptions, $N((V,\hom),\xi)=(\V,\hom_\xi)$ is a
$\TV$-category; this structure has a crucial role in our study, as we will see in the next section.

\subsection{$\TV$-bimodules}\label{sub:3.5}
Given $\TV$-categories $(X,a)$ and $(Y,b)$, a \emph{$\TV$-bimodule} (or simply a \emph{bimodule}) $\psi:(X,a)\modto
(Y,b)$ is a $\V$-relation $\psi:TX\rel Y$ such that $\psi\circ a\leq\psi$ and $b\circ\psi\leq\psi$, where the composition
$s\circ r$ of two $\V$-relations $r:TX\rel Y$ and $s:TY\rel Z$ is given by the \emph{Kleisli convolution} (see
\cite{HT}), that is
\[s\circ r=s\cdot Tr\cdot m_X^\circ.\] Under our assumptions bimodules compose, with the $\TV$-categorical structures as
identities for this composition. We denote by $\TVMod$ the category of $\TV$-categories and $\TV$-bimodules. $\TVMod$ is
locally preordered by the preorder inherited from $\VRel$.

Every $\TV$-functor $f:(X,a)\to(Y,b)$ induces a pair of bimodules $f_*:(X,a)\modto(Y,b)$ and $f^*:(Y,b)\modto(X,a)$,
defined by $f_*=b\cdot Tf$ and $f^*=f^\circ\cdot b$; that is, $f_*(\fx,y)=b(Tf(\fx),y)$ and $f^*(\fy,x)=b(\fy,f(x))$, for
$\fx\in TX$, $\fy\in TY$, $x\in X$ and $y\in Y$. The Kleisli convolution becomes simpler when composing with these
bimodules: for any $\varphi:X\modto Z$ and $\psi:Z\modto X$, $f^*\circ\varphi=f^\circ\cdot\varphi$ and $\psi\circ
f_*=\psi\cdot Tf$. It is easy to check that $a\leq f^*\circ f_*$ and $f_*\circ f^*\leq b$, that is, $f_*\dashv f^*$. The
$\TV$-functor $f$ is said to be \emph{fully faithful} when $f^*\circ f_*=a$, or, equivalently,
$a(\fx,x)=b(Tf(\fx),f(x))$, for every $\fx\in TX, x\in X$. The local (pre)order on $\TVCat$ corresponds to the local
(pre)order on $\TVMod$: for $\TV$-functors $f,g:(X,a)\to(Y,b)$,
\[f\leq g\iff f^*\leq g^*\iff f_*\geq g_*.\]

\section{The presheaf monad and its submonads}

\subsection{The Yoneda Lemma}
The tensor product in $\V$ defines a tensor product in $\TVCat$, with $(X,a)\otimes(Y,b)=(X\times Y,c)$, where
$c(\fw,(x,y))=a(T\pi_1(\fw),x)\otimes b(T\pi_2(\fw),y)$, for $\fw\in T(X\times Y)$, $x\in X$, $y\in Y$. Its neutral
element is $E=(1,e_1^\circ)$. For each $\TV$-category $(X,a)$, the functor $X^\op\otimes(\;):\TVCat\to\TVCat$ has a right
adjoint $(\;)^{X^\op}:\TVCat\to\TVCat$.
\begin{upropC}[\cite{CH09}]\label{th:bimod}
For $\TV$-categories $(X,a), (Y,b)$ and a $\V$-relation $\psi:TX\rel Y$, the following conditions are equivalent:
\begin{enumerate}[label=(\em\roman*)]
\item $\psi:(X,a)\modto(Y,b)$ is a bimodule;
\item $\psi:(X,a)^\op\otimes (Y,b)\to(\V,\hom_\xi)$ is a $\TV$-functor.
\end{enumerate}
\end{upropC}
Since $a:(X,a)\modto(X,a)$ is a bimodule, this result tells us that $a:X^\op\otimes X\to\V$ is a $\TV$-functor, and
therefore, from the adjunction $X^\op\otimes (\;)\dashv (\;)^{X^\op}$, $a$ induces the \emph{Yoneda $\TV$-functor}
\[\xymatrix{(X,a)\ar[rr]^{\yoneda_X}&& \V^{X^\op}}.\]
The following result provides a Yoneda Lemma for $\TV$-categories.
\begin{uthmC}[\cite{CH09}]
Let $(X,a)$ be a $\TV$-category. For all $\psi\in\V^{X^\op}$ and all $\fx\in TX$,
\[\widehat{a}(T\yoneda_X(\fx),\psi)=\psi(\fx),\]
where $\widehat{a}$ denotes the $\TV$-categorical structure on $\V^{X^\op}$. In particular, $\yoneda_X$ is fully
faithful.
\end{uthmC}

\subsection{The presheaf monad}
In order to work in an $\Ord$-enriched category, from now on we restrict ourselves to $(\mT,\V)$-$\Cat_0$. We remark that
the results of the previous subsection remain valid when we replace $(\mT,\V)$-\Cat\ by $(\mT,\V)$-$\Cat_0$.
Denoting $\V^{X^\op}$ by $PX$, we point out that, via Theorem \ref{th:bimod},
\[PX=\{\varphi:(X,a)\modto E\,|\,\varphi \mbox{ bimodule}\}.\]
Moreover, the Yoneda $\TV$-functor $\yoneda_X$ turns out to assign to each $x\in X$, that is, to each
$\TV$-functor $x:E\to X$, the bimodule $x^*:X\modto E$. Each $\TV$-functor $f:(X,a)\to(Y,b)$ induces a $\TV$-functor
$Pf:PX\to PY$, assigning to $\varphi:X\modto E$ the bimodule $\varphi\circ f^*:Y\modto E$, that is $Pf=(\;)\circ f^*$.
This defines an endofunctor $P$ on $\TVCat$. From the adjunction $f_*\dashv f^*$, for every $\TV$-functor
$f:(X,a)\to(Y,b)$ one gets a right adjoint to $Pf$, $P^*f=(\;)\circ f_*:PY\to PX$. In particular, $P\yoneda_X:PX\to PPX$
has a right adjoint $\yonmult_X:PPX\to PX$, which, together with $P$ and $\yoneda$, defines a lax idempotent monad, the
\emph{presheaf monad}. Next we show that this monad is simple. In order to do that we use Proposition \ref{sub:2.5}.

\begin{uthm}
The presheaf monad $\mP$ on $\TVCat$ is simple.
\end{uthm}
\begin{proof}
We need to show that, for any $\TV$-functor $f:(X,a)\to(Y,b)$, in the diagram below $PLf\dashv \yonmult_X\cdot Pq_f$.
\[\xymatrix{X\ar[rrr]^{\yoneda_X}\ar[ddd]_f\ar[rd]^{Lf}&&&PX\ar[ddd]^(0.6){Pf}\ar[rd]^{PLf}\ar@<1ex>[rrr]^{P\yoneda_{X}}&&
&PPX\ar@{}[lll]|{\bot}\ar@<1ex>[lll]^{\yonmult_X}\\
&Kf\ar@{}[rrdd]|{\geq}\ar[ldd]^{Rf}\ar[rru]_{q_f}\ar[rrr]_(0.4){\yoneda_{Kf}}&&&PKf\ar[ldd]^{PRf}\ar[rru]_{Pq_f}\\
\\
Y\ar[rrr]_{\yoneda_Y}&&&PY}\]
First we recall that the comma object $(Kf,\tilde{a})=Pf\downarrow \yoneda_Y$ is given by $Kf=\{(\varphi,y)\in PX\times
Y\,|\,Pf(\varphi)\leq y^*\}$, and $\tilde{a}(\fw,(\varphi,y))=\widehat{a}(Tq_f(\fw),\varphi)\wedge b(TRf(\fw),y)$, where
$\widehat{a}$ is the structure on $PX$.
On one hand, as we observed before, $PLf$ has as right adjoint the $\TV$-functor $P^*Lf=(\;)\circ(Lf)_*$. On the other
hand, $\yonmult_X\cdot Pq_f=P^*\yoneda_X\cdot Pq_f=(\;)\circ q_f^*\circ(\yoneda_X)_*$. Next we will show that
$(Lf)_*=q_f^*\circ (\yoneda_X)_*:X\modto Kf$, which concludes the proof.
For each $\fx\in TX$ and $(\varphi,y)\in Kf$,
\[(q_f^*\circ(\yoneda_X)_*)(\fx,(\varphi,y))=\widehat{a}(T\yoneda_X(\fx),\varphi)=\varphi(\fx),\]
while
\[(Lf)_*(\fx,(\varphi,y))=\tilde{a}(TLf(\fx),(\varphi,y))=\widehat{a}(T\yoneda_X(\fx),\varphi)\wedge b(Tf(\fx),y)=\varphi(\fx)\wedge b(Tf(\fx),y).\]
To show that $b(Tf(\fx),y)\geq\varphi(\fx)$, so that $(Lf)_*(\fx,(\varphi,y))=\varphi(\fx)$, note that
\[\begin{array}{rcll}
b(Tf(\fx),y)&=&\widehat{b}(T\yoneda_X(Tf(\fx)),y^*)&\mbox{(by the Yoneda Lemma)}\\
&\geq&\widehat{b}(T\yoneda_X(Tf(\fx)),Pf(\varphi))&\mbox{(by definition of $Kf$)}\\
&=&\widehat{b}(TPf\cdot T\yoneda_X(\fx),Pf(\varphi))\\
&\geq&\widehat{a}(T\yoneda(\fx),\varphi)&\mbox{(because $Pf$ is a $\TV$-functor)}.\hspace{1.96cm}\qEd
\end{array}\]
\def\popQED{}
\end{proof}

\begin{uprop}\label{prop:fff}
\begin{enumerate}[label=(\em\arabic*)]
\item A $\TV$-functor is a $P$-embedding if, and only if, it is fully faithful.
\item Fully faithful $\TV$-functors are pullback stable.
\end{enumerate}
\end{uprop}
\begin{proof} \leavevmode
  \begin{enumerate}
\item If $f:(X,a)\to(Y,b)$ is a $\TV$-functor, then $Pf$ has a right adjoint, $P^*f$. It remains to show that $P^*f\cdot
Pf=1_X$ when $f$ is fully faithful; this means $f^*\cdot f_*=a$, and so, for any bimodule $\varphi:X\modto E$,
\[P^*f\cdot Pf(\varphi)=\varphi\circ f^*\circ f_*=\varphi\circ a=\varphi.\]
Conversely, if $P^*f\cdot Pf=1_{PX}$, then, for any $x\in X$,
\[a(-,x)=x^*=P^*f\cdot Pf(x^*)=x^*\circ f^*\circ f_*=f^\circ\cdot x^\circ\cdot b\cdot Tf=b(Tf(-),f(x)),\]
that is, $f$ is fully faithful.

\item As in any topological category, (bijective, fully faithful $\TV$-functors) is an orthogonal factorisation system in
$\TV$-\Cat, and therefore fully faithful $\TV$-functors are pullback-stable.
\qedhere
\end{enumerate}
\end{proof}

\subsection{Presheaf submonads}\label{sub:4.3}
Let $\Phi$ be a class of $\TV$-bimodules satisfying the conditions:
\begin{enumerate}[leftmargin=8mm]
\item[(S1)] $\Phi$ is closed under composition.
\item[(S2)] For every $\TV$-functor $f$, $f^*\in\Phi$.
\item[(S3)] For every $\TV$-bimodule $\psi:X\modto Y$, $\psi\in\Phi$ provided that $y^*\circ\psi\in\Phi$ for every
    $y\in Y\!$.
\end{enumerate}
We call such a class \emph{saturated}. There is a largest saturated class, of all $\TV$-bimodules, and a smallest one,
$\{f^*\,|\,f \mbox{ is a $\TV$-functor}\}$. In the last section we will present other saturated classes.

For each $\TV$-category $(X,a)$, we define
\[\Phi X=\{\varphi:X\modto E\,|\,\varphi\in\Phi\}\subseteq PX,\]
equipped with the structure $\widehat{a}$ inherited from $PX$, and, to each $\TV$-functor $f:(X,a)\to(Y,b)$ we assign
\[\Phi f:\Phi X\to\Phi Y, \mbox{ with }\Phi f(\varphi)=\varphi\circ f^*.\]
Since $x^*\in\Phi$ for every $x\in X$, $\yoneda_X$ corestricts to $\Phi X$,
\[\xymatrix{X\ar[rr]^{\yoneda_X^\Phi}&&\Phi X}.\]
Moreover, condition (S3) guarantees that $\yonmult_X=P^*\yoneda_X$ (co)restricts to $\yonmult_X^\Phi:\Phi\Phi X\to\Phi
X$: by the Yoneda Lemma, for all $\varphi\in \Phi X$, $\varphi^*\circ P^*\yoneda_X=\varphi\in\Phi$. So, $(\Phi,
\yoneda^\Phi,\yonmult^\Phi)$ is a submonad of $P$.

\begin{uthm}\label{th:submon}
If $\Phi$ is a saturated class of bimodules, then the monad $(\Phi,\yoneda^\Phi,\yonmult^\Phi)$ is lax idempotent and
simple, and so it defines a lax orthogonal factorisation system.
\end{uthm}
\begin{proof}
Since fully faithful $(\mT,\V)$-functors are pullback-stable and full, and the inclusion $\Phi X\to PX$ is clearly fully
faithful, this result follows directly from Theorem \ref{sub:2.6}.
\end{proof}

\section{Examples: The induced \textsc{lofs}s}

\subsection{General description}
Now let us fix a saturated class $\Phi$ of $\TV$-bimodules as in \ref{sub:4.3}. The presheaf submonad $\Phi$ induces a
\textsc{lofs} $(\mL^\Phi,\mR^\Phi)$, and consequently a \textsc{wfs} $(\cL^\Phi,\cR^\Phi)$ where $\cL^\Phi$ is the class
of $\Phi$-embeddings.

Following \cite{CH09}, we say that a $\TV$-functor $f$ is \emph{$\Phi$-dense} if $f_*\in\Phi$.

\begin{ulemC}[\cite{CH09}]
For a $\TV$-functor $h$, the following conditions are equivalent:
\begin{enumerate}[label=(\em\roman*)]
\item $h$ is $\Phi$-dense;
\item $\Phi h$ is a left adjoint;
\item $\Phi h$ is $\Phi$-dense.
\end{enumerate}
\end{ulemC}

We note that $\Phi h$ has a right adjoint if and only if the right adjoint $P^*h$ of $Ph$ can be (co)restricted to
$\Phi_* h:\Phi Y\to\Phi X$, which is the case precisely when $h_*\in\Phi$.
\begin{uprop}
For a $\TV$-functor $h:(X,a)\to(Y,b)$, the following conditions are equivalent:
\begin{enumerate}[label=(\em\roman*)]
\item $h$ belongs to $\cL^\Phi$;
\item $h$ is fully faithful and $\Phi$-dense.
\end{enumerate}
\end{uprop}

\begin{proof}
(i)$\Rightarrow$(ii): From Theorem \ref{th:submon} we know that a $\Phi$-embedding $h$ is fully faithful, and, by
definition, $\Phi h$ is a left adjoint. (ii)$\Rightarrow$(i): If $h$ is $\Phi$-dense, then $\Phi h$ has a right adjoint
$\Phi^* h$, and so it remains to show that, when $h^*\circ h_*=a$, $\Phi^* h\cdot\Phi h=1_{PX}$: since $x^*\in\Phi$ for
every $x\in X$, the proof follows the arguments used in Proposition \ref{prop:fff}(1).
\end{proof}

\begin{ucor}
For every $\TV$-category $(X,a)$, $\yoneda_X^\Phi$ is a $\Phi$-embedding.
\end{ucor}

The class $\cR^\Phi$ is the
class of morphisms with the right lifting property with respect to morphisms in $\cL^\Phi$; that is, a morphism belongs to $\cR^\Phi$ if, and only if, it is \emph{injective with respect to the class} $\cL^\Phi$. Since $(\mL^\Phi,\mR^\Phi)$ is a \textsc{lofs},
these morphisms have the \textsc{kz}-lifting property with respect to morphisms in $\cL^\Phi$.
Such morphisms encompass interesting
properties, but are usually very difficult to identify.

\subsection{Examples: the lari--opfibration {\normalfont\textsc{lofs}}}
When \[\Phi=\{f^*\,|\,f\mbox{ is a $\TV$-functor}\},\]
then $\yoneda^\Phi_X:X\to\Phi X$, $x\mapsto x^*$, is an isomorphism; that is, the monad $(\Phi,\yoneda^\Phi,\yonmult^\Phi)$ is isomorphic to the identity monad. Therefore the corresponding \textsc{lofs} $(\mL,\mR)$ is the one studied in \cite[Examples 4.7, 4.8]{CLF17}, and the monad $\mR$ is the free (split) opfibration monad on $\TV$-$\Cat_0$. Then $\cL^\Phi$ is the class of laris and $\cR^\Phi$ the class of split opfibrations.

\subsection{Examples: the presheaf {\normalfont\textsc{lofs}}}
Let us now take the largest saturated class
\[\Phi=PX=\{\varphi:(X,a)\modto E\,|\,\varphi \mbox{ bimodule}\}.\]
From Theorem \ref{sub:2.5} we know that the presheaf monad defines a \textsc{lofs} $(\mL,\mR)$ in $(\mT,\V)$-$\Cat_0$,
and, consequently, a \textsc{wfs} $(\cL,\cR)$, where $\cL$ is the class of fully faithful $(\mT,\V)$-functors. It is easy
to check that they coincide with extremal monomorphisms in $(\mT,\V)$-$\Cat_0$, that is, \emph{topological embeddings}.
Therefore, from Theorem \ref{sub:2.5} we conclude that, for every quantale $\V$ and monad $\mT$ in the conditions of
\ref{sub:3.1}, $(\mT,\V)$-$\Cat_0$ has a \textsc{wfs} $(\cL,\cR)$ where $\cL$ is the class of embeddings.
\medskip

When $\mT$ is the identity monad and $\V=\two$, that is, in the category $\Ord$ of (anti-symmetric) ordered sets and monotone maps, the morphisms in $\cR$ were characterised by Ad\'{a}mek (as mentioned in \cite{Th00}), as those monotone maps which are \emph{fibre-complete}, \emph{fibrations} and \emph{co-fibrations} (among some other characterisations; see also
\cite{AHRT}).
\medskip

When $\mT$ is the identity monad and $\V=[0,\infty]_+$, that is, in the category $\Met_0$ of separated generalised metric spaces and non-expansive maps,
a characterisation of the morphisms in $\cR$ is not known. It follows from \cite{hyp} that a (non-expansive) map $f:X\to 1$ belongs to $\cR$ if, and only if, $X$ is an \emph{hyperconvex metric space} (see also \cite{Is}).
\medskip

When $\mT=\mU$ and $\V=\two$, that is, in the category $\Top_0$ of \textsc{T0}-spaces and continuous maps, morphisms in $\cR$ were studied in a series of papers by Cagliari and Mantovani (see \cite{CM} and references there), and characterised in \cite{CCM, CCM2}; in \cite{CCM} they are identified via a way-below relation while \cite{CCM2} gives characterisations that extend those of $\Ord$ mentioned above.

\subsection{Examples: the Lawvere {\normalfont\textsc{lofs}}}

The choice of \[\Phi=\{\psi\in PX\,|\,\psi\mbox{ is right adjoint}\}\] has particular relevance.
\medskip

When $\mT=\mId$ and $\V=[0,\infty]_+$, the injective objects in $\Met_0$ with respect to
$\Phi$-embeddings are the Cauchy-complete metric spaces, that is, a non-expansive map $X\to 1$ belongs to $\cR^\Phi$ if
and only if $X$ is Cauchy-complete (see \cite{Law73, CH09, CH08}). Therefore, the morphisms in $\cR^\Phi$ are good candidates for a fibrewise notion of
Cauchy-completeness. This \textsc{lofs} was studied in \cite{CLF17}. The morphisms in $\cL^\Phi$ are the embeddings (=isometries) $f:(X,a)\to(Y,b)$ such that, for every $y\in Y$, $b(f(-),y)=\lim_na(-,x_n)$ for some Cauchy sequence $(x_n)$ in $X$.
We point out that the non-expansive maps in $\cR^\Phi$ do not coincide with Sozubek's L-complete maps \cite{So}. Indeed,
Sozubek defines them via an injective property, but his left part -- the so called $L$-equivalences -- is a proper
subclass of $\cL^\Phi$.
\medskip

When $\mT=\mU$ and $\V=\two$, this choice of $\Phi$ gives also an interesting \textsc{lofs}. As shown in \cite{CH09}, the $\Phi$-algebras for this monad in $\Top_0$ are the sober spaces. Since sober spaces are also the algebras for the lax idempotent and simple monad of completely prime filters on $\Top_0$, the \textsc{wfs} induced by $\Phi$ coincides with the \textsc{wfs} studied in \cite{CCM}; that is, $\cL^\Phi$ is the class of \emph{completely flat embeddings} and $\cR^\Phi$ is the class of \emph{fibrewise sober} continuous maps (see \cite{Es98, EF99, RV}).

\subsection{Further examples}
Using the techniques of \cite{Es98, EF99} and \cite[3.7]{CH08}, one can define saturated classes of $\TV$-bimodules
$\Phi_0$ and $\Phi_\omega$ so that the left parts of the corresponding \textsc{wfs} are
$\cL^{\Phi_0}=\{$dense embeddings$\}$ and $\cL^{\Phi_\omega}=\{$flat embeddings$\}$. The simple presheaf submonads they define induce \textsc{lofs} whose underlying \textsc{wfs} were
studied in \cite{CCM}, where $\cR^{\Phi_0}$ and $\cR^{\Phi_\omega}$ give the fibrewise notions of
\emph{Scott domains} and \emph{stably compact spaces}.

{


\begin{thebibliography}{99}
\bibitem{AHRT} J. Ad\'{a}mek, H. Herrlich, J. Rosick\'y, W. Tholen, Weak factorization systems and topological functors.
    \emph{Appl. Categ. Structures} 10 (2002) 237--249.
    \bibitem{hyp} N. Aronszajn, P. Panitchpakdi, Extension of uniformly continuous transformations and
hyperconvex metric spaces. \emph{Pacific J. Math.} 6 (1956) 405--439.
\bibitem{Ba} M. Barr,  Relational algebras. In: \emph{Lecture Notes in Math.} 137, Springer, Berlin 1970,
pp. 39--55.
\bibitem{BG} J. Bourke, R. Garner, Algebraic weak factorisation systems: I. \emph{J. Pure Appl. Algebra} 220 (2016)
    108--147.
\bibitem{CCM} F. Cagliari, M.M. Clementino, S. Mantovani, Fibrewise injectivity and Kock-Z\"oberlein monads. \emph{J.
    Pure Appl. Algebra} 216 (2012) 2411--2424.
\bibitem{CCM2} F. Cagliari, M.M. Clementino, S. Mantovani, Fibrewise injectivity in order and topology. \emph{Topology
    Appl.} 200 (2016) 61--78.
    \bibitem{CM} F. Cagliari, S. Mantovani, Injectivity and sections. \emph{J. Pure Appl. Algebra} \textbf{204} (2006) 79--89.
\bibitem{CHK} C. Cassidy, M. H\'{e}bert, G.M. Kelly, Reflective subcategories, localizations and factorization systems.
    \emph{J. Aust. Math. Soc. Ser. A} 38 (1985) 287--329.
\bibitem{CCH} D. Chikhladze, M.M. Clementino, D. Hofmann, Representable $\TV$-categories. \emph{Appl. Categ. Structures}
    23 (2015) 829--858.
\bibitem{CH03} M.M. Clementino, D. Hofmann, Topological features of lax algebras. \emph{Appl. Categ. Structures} 11
    (2003) 267--286.
\bibitem{CH08} M.M. Clementino, D. Hofmann, Relative injectivity as cocompleteness for a class of distributors.
    \emph{Theory Appl. Categ.} 21 (2008) 210--230.
\bibitem{CH09} M.M. Clementino, D. Hofmann, Lawvere completeness in topology. \emph{Appl. Categ. Structures} 17 (2009)
    175--210.

\bibitem{CHJ14} M.M. Clementino, D. Hofmann, G. Janelidze, The monads of classical algebra are seldom weakly cartesian.
    \emph{J. Homotopy Relat. Struct.} 9 (2014) 175--197.


\bibitem{CLF16} M.M. Clementino, I. L\'{o}pez-Franco, Lax orthogonal factorisation systems. \emph{Adv. Math.} 302 (2016)
    458--528.
\bibitem{CLF17} M.M. Clementino, I. L\'{o}pez-Franco, Lax orthogonal factorisations in ordered structures, arXiv: 1702.02602, 2017.

\bibitem{CT03} M.M. Clementino, W. Tholen, Metric, Topology and Multicategory -- A Common Approach. \emph{J. Pure Appl.
    Algebra} 179 (2003) 13--47.

\bibitem{Es98} M. Escard\'{o}, Properly injective spaces and function spaces. \emph{Topology Appl.} 89 (1998) 75--120.

\bibitem{EF99} M. Escard\'{o}, R. Flagg, Semantic domains, injective spaces and monads. \emph{Electr. Notes in Theor. Comp.
    Science} 20, electronic paper 15 (1999).
\bibitem{Ga} R. Garner, Understanding the small object argument. \emph{Appl. Categ. Structures} 17 (2009) 247--285.

\bibitem{GT} M. Grandis, W. Tholen, Natural weak factorization systems. \emph{Arch. Math.} (Brno) 42 (2006) 397--408.
\bibitem{Gr} J.W. Gray, Fibred and cofibred categories. In: 1966 Proc. Conf. Categorical Algebra (La Jolla, Calif., 1965) pp. 21--83 Springer, New York.
\bibitem{H07} D. Hofmann, Topological theories and closed objects. \emph{Adv. Math.} 215 (2007), 789--824.
\bibitem{H11} D. Hofmann, Injective spaces via adjunction. \emph{J. Pure Appl. Algebra} 215 (2011), 283--302.

\bibitem{HST} D. Hofmann, G. Seal, W. Tholen (eds.), \emph{Monoidal Topology. A categorical approach to order, metric,
    and topology}. Encyclopedia of Mathematics and its Applications, 153. Cambridge University Press, Cambridge, 2014.
\bibitem{HT} D. Hofmann, W. Tholen, Kleisli compositions for topological spaces. \emph{Topology Appl.} 153 (2006)
    2952--2961.
\bibitem{Is} J. Isbell, Six theorems about injective metric spaces. \emph{Comment. Math. Helv.} 39 (1964) 65--76.
\bibitem{Ko} A. Kock, Monads for which structures are adjoint to units. \emph{J. Pure Appl. Algebra} 104 (1995) 41--59.

\bibitem{Law73} F.W. Lawvere, Metric spaces, generalized logic, and closed
categories. {\em Rend. Sem. Mat. Fis. Milano} {\bf 43} (1973)
135--166.
\bibitem{LF} I. L\'{o}pez Franco, Cofibrantly generated lax orthogonal factorisation systems. arXiv 1510.07131, 2015.

\bibitem{Lo} R. Lowen, {\em Approach Spaces: The missing
    link in the Topology-Uniformity-Metric Triad}. Oxford
    Mathematical Monographs (Oxford University Press, Oxford 1997).

\bibitem{Ma02} E.G. Manes, Taut monads and T0-spaces. \emph{Theoret. Comput. Sci.} 275 (2002), 79--109.

\bibitem{Qui} D.G. Quillen, \emph{Homotopical Algebra}. Lecture Notes in Mathematics, vol. 43. Springer, 1967.
\bibitem{RV} G. Richter, A. Vauth, Fibrewise sobriety. In: \emph{Categorical structures and their applications}, World
    Sci. Publ. (2004), pp. 265--283.

\bibitem{So} S. Sozubek, Lawvere completeness as a topological property. \emph{Theory Appl. Categ.} 27 (2013) 242--262.

\bibitem{Th00} W. Tholen, Injectives, exponentials, and model categories. In \emph{Abstracts of
the International Summer Conference in Category Theory}, Como, Italy, 2000, pp. 183--190.
\end{thebibliography}
\end{document}